\documentclass[12pt]{amsart}
\usepackage[latin1]{inputenc}
\usepackage[english]{babel}

\usepackage{indentfirst}
\usepackage{amssymb}
\usepackage{amsthm}

\newcommand{\FF}{\protect{\mathcal F}}
\newcommand{\DD}{\protect{\mathcal D}}

\newcommand{\vf}{\varphi}
\newcommand{\sm}{\setminus}
\newcommand{\sub}{\subseteq}

\def\dfrac#1#2{{\displaystyle{#1\over#2}}}

\newcommand\eps{\ensuremath{\varepsilon}}

\newcommand{\Ba}{{\rm Ba}}

\newcommand{\lin}{\protect{\rm span}}
\newcommand{\alg}{\mathfrak A}
\newcommand{\algb}{\mathfrak B}
\newcommand{\algc}{\mathfrak C}

\newcommand{\UU}{{\mathcal U}}

\newcommand{\NN}{{\mathcal N}}
\newcommand{\PP}{{\mathcal P}}

\newcommand{\btu}{\bigtriangleup}

\newcommand{\er}{\mathbb R}

\newcommand{\ult}{{\rm ULT}}

\newcommand{\clop}{\protect{\rm Clop} }

\newtheorem{theo}{Theorem}[section]
\newtheorem{lem}[theo]{Lemma}
\newtheorem{cor}[theo]{Corollary}

\newtheorem{prop}[theo]{Proposition}

\newtheorem{lemma}[theo]{Lemma}

\newtheorem{problem}[theo]{Problem}

\theoremstyle{definition}
\newtheorem{definition}[theo]{Definition}

\theoremstyle{remark}
\newtheorem{remark}[theo]{Remark}
\numberwithin{equation}{section}

\def\epsilon{\varepsilon}

\providecommand{\MR}{\relax\ifhmode\unskip\space\fi MR }

\providecommand{\href}[2]{#2}

\title{On Baire measurability in spaces of continuous functions}

\author{ A. Avil\'{e}s}
\address{Departamento de Matem\'{a}ticas\\
Facultad de Matem\'{a}ticas\\ Universidad de Murcia\\ 30100 Espinardo (Murcia)\\
Spain} \email{avileslo@um.es}

\author{G. Plebanek}
\address{Instytut Matematyczny\\ Uniwersytet Wroc\l awski\\ Wroc\l aw\\ Poland} \email{grzes@math.uni.wroc.pl}

\author{J. Rodr\'{i}guez}
\address{Departamento de Matem\'{a}tica Aplicada\\
Facultad de Inform\'{a}tica\\ Universidad de Murcia\\ 30100 Espinardo (Murcia)\\
Spain} \email{joserr@um.es}



\thanks{A.~Avil\'{e}s and J.~Rodr\'{i}guez were supported by MEC (Project MTM2011-25377). 
A. Avil\'{e}s was supported by {\em Ramon y Cajal} contract (RYC-2008-02051).
G. Plebanek was supported by MNiSW Grant N N201 418939 (2010--2013)}

\begin{document}
\begin{abstract}
Let $C(K)$ be the Banach space of all continuous functions on a given compact space~$K$. 
We investigate the $w^*$-sequential closure in~$C(K)^*$ of the set of all finitely supported probabilities on~$K$. We
discuss the coincidence of the Baire $\sigma$-algebras on~$C(K)$ associated to the weak and pointwise convergence topologies. 
\end{abstract}

\maketitle

\section{Introduction}

We denote by~$\omega$ the set of all natural numbers $\{0,1,2,\dots\}$. 
Any $n \in \omega$ is often regarded as the set $\{0,1,\dots,n-1\}$.

Let $K$ be a compact space (all our topological spaces are Hausdorff), let $C(K)$ be
the Banach space of all continuous real-valued functions on~$K$ and let
$M(K)=C(K)^*$ be the space of all Radon (signed) measures on~$K$. Throughout the paper $M(K)$ 
is equipped with the weak$^*$ topology (denoted by $w^*$ for short) unless otherwise stated.
We denote by $M^+(K)$ (resp. $P(K)$) the subset of~$M(K)$ made up of all
Radon non-negative (resp. probability) measures on~$K$. For every $t\in K$ 
we denote by $\delta_t\in P(K)$ the Dirac measure at~$t$. We shall write 
${\rm co}\Delta_K$ for the convex hull 
of the set $\Delta_K := \{\delta_t: \, t\in K\}$ in $M(K)$.
Given a set $A\sub M(K)$, we denote by $Seq(A)$ the sequential closure of $A$ in~$M(K)$, 
that is, the smallest subset of $M(K)$ that contains~$A$ and is closed under limits of $w^*$-convergent sequences. 
The sequential closure is obtained by a transfinite procedure as follows. Define $Seq^0(A):=A$, and
let $Seq^{\alpha+1}(A)$ be the set of all limits of $w^*$-convergent sequences in 
$Seq^{\alpha}(A)$, and let $Seq^\alpha(A) := \bigcup_{\beta<\alpha}Seq^\beta(A)$ whenever~$\alpha$ is a limit ordinal. 
Then $Seq(A) = Seq^{\omega_1}(A)$, where $\omega_1$ stands for the first uncountable ordinal. 

The set ${\rm co}\Delta_K$ is $w^*$-dense in $P(K)$ (just apply the Hahn-Banach theorem).  
For an arbitrary $\mu\in P(K)$, a classical result (see~\cite{nie}) states that 
$\mu \in Seq^1({\rm co}\Delta_K)$ if and only if $\mu$ admits a {\em uniformly distributed sequence}, i.e. a 
sequence~$\{t_n\}_{n\in \omega}$ in~$K$ such that 
$\{\frac{1}{n}\sum_{i<n} \delta_{t_i}\}_{n\in \omega}$ is $w^*$-convergent to~$\mu$.
There is a number of well-studied classes of compact spaces $K$ on which
{\em every} Radon probability measure admits a uniformly distributed sequence or, equivalently,
the equality
\begin{equation}\label{equation:Seq1}
	Seq^1({\rm co}\Delta_K)=P(K)
\end{equation}
holds true. Indeed, $K$ has such a property whenever it is
metrizable, Eberlein, Rosenthal, Radon-Nikod\'{y}m or
a totally ordered compact line (see \cite{mer-J} and the references therein). The space
$K=2^\mathfrak{c}$ enjoys that property as well \cite[491Q]{freMT-4}, where $\mathfrak{c}$ stands for the cardinality of the continuum.
The Stone space $K$ of a minimally generated Boolean algebra satisfies $Seq({\rm co}\Delta_K) = P(K)$ (see~\cite{bor-ple}) and, 
in fact, this result can be strengthen to saying that equality~\eqref{equation:Seq1} holds.

Under the Continuum Hypothesis, we present in Section~\ref{ex1} a construction of a compact $0$-dimensional space $K$ such that
$$ 
	Seq^1({\rm co}\Delta_K) \neq Seq({\rm co}\Delta_K) = P(K)
$$
(see Theorem~\ref{ex1:9}). Our example has some features of an $L$-space
constructed in~\cite{Ku81} and related constructions given in~\cite{Pl97}. 
In fact, the compact space~$K$ of~Theorem~\ref{ex1:9} satisfies 
$Seq^1({\rm co}\Delta_K) \neq Seq^2({\rm co}\Delta_K)$ and $Seq^3({\rm co}\Delta_K) = P(K)$.
Along this way, it was recently proven in~\cite{bor-sel} (without additional set-theoretic assumptions) that for every ordinal $1\leq \alpha<\omega_1$ there is
a compact space~$K^{(\alpha)}$ such that 
$$
	Seq^\alpha({\rm co}\Delta_{K^{(\alpha)}})\sm \bigcup_{\beta<\alpha}Seq^\beta({\rm co}\Delta_{K^{(\alpha)}})\neq \emptyset
$$
and $Seq^{\alpha+1}({\rm co}\Delta_{K^{(\alpha)}})=Seq({\rm co}\Delta_{K^{(\alpha)}})\neq P(K^{(\alpha)})$.

Our interest on these questions is somehow motivated by their connection with the study of
Baire measurability in the space $C(K)$. Namely, 
if $C_p(K)$ (resp. $C_w(K)$) stands for $C(K)$ equipped with the pointwise convergence (resp. weak) topology, then
the corresponding Baire $\sigma$-algebras satisfy
$$
	\Ba(C_p(K)) \sub \Ba(C_w(K)).
$$
It is well-known (see~\cite{edg-J}) that $\Ba(C_p(K))$ is generated by $\Delta_K$,
while $\Ba(C_w(K))$ is generated by $P(K)$. Thus, the equality
\begin{equation}\label{equation:coincidence}
	\Ba(C_p(K)) = \Ba(C_w(K))
\end{equation}
holds true whenever $Seq({\rm co}\Delta_K) = P(K)$, and this is the case for many spaces
as we pointed out above. The compact space of Theorem~\ref{ex1:9} makes clear
that equalities \eqref{equation:Seq1} and~\eqref{equation:coincidence} are not equivalent. In Section~\ref{section:betaN}
we pay further attention to~\eqref{equation:coincidence} and 
show that it fails for $K=\beta\omega$ and $K=\beta\omega\sm\omega$ (Theorem~\ref{neq_betaN} and
Corollary~\ref{neq_betaNsmN}). Some related open problems are posed at the end of the paper.

\subsection*{Terminology}

We write $\mathcal{P}(S)$ to denote the power set of any set~$S$.
Given a Boolean algebra~$\alg$, by a `measure' on~$\alg$ we mean a 
bounded finitely additive measure.
The Stone space of all ultrafilters on~$\alg$ is denoted by~$\ult(\alg)$. Recall that the Stone isomorphism
between $\alg$ and the algebra $\clop(\ult(\alg))$ of clopen subsets of $\ult(\alg)$ is given by
$$
	\alg \to \clop(\ult(\alg)), \quad A \mapsto \widehat{A}=\{\FF\in\ult(\alg): \, A\in\FF\}.
$$
Every measure $\mu$ on~$\alg$ induces a measure $\widehat{A} \mapsto \mu(A)$ on $\clop(\ult(\alg))$ which can be uniquely 
extended to a Radon measure on~$\ult(\alg)$ (see e.g. \cite[Chapter~5]{semadeni}); such Radon measure is still denoted by the same letter~$\mu$.
We shall need the following useful fact about extensions of measures. 

\begin{lem}[\cite{Lipecki74,Plachky76}]\label{extensions_of_measures}
Let $\algc\supseteq \algb$ be Boolean algebras and let $\mu$ be a measure on~$\algb$. Then
$\mu$ can be extended to a measure $\nu$ on~$\algc$ such that
$\inf \{\nu(C\btu B): B\in\algb\}=0$ for every $C\in\algc$.
\end{lem}

\section{A compact space $K$ such that $Seq^1({\rm co}\Delta_K) \neq Seq({\rm co}\Delta_K) = P(K)$}\label{ex1}

For the sake of the construction we first note the following two lemmas.
We denote by $\lin\Delta_K$ the linear span of~$\Delta_K$ in~$M(K)$. 

\begin{lem}\label{seq_abs_cont2}
Let $K$ be a compact space and let $\mu\in Seq^{\alpha}(\lin \Delta_K)$ for some $\alpha<\omega_1$.
If $\vf\in C(K)$ and $\nu\in M(K)$ is defined by  
$$
	\nu(\Omega):=\int_\Omega \vf\, {\rm d}\mu \quad
	\mbox{for every Borel set }\Omega \sub K,
$$
then $\nu \in Seq^{\alpha}(\lin\Delta_K)$ as well. The same statement holds if $\lin\Delta_K$
is replaced by $M^+(K)\cap \lin\Delta_K$ and $\vf \geq 0$.
\end{lem}
\begin{proof} We proceed by transfinite induction. The case $\alpha=0$ being obvious,
suppose that $1\leq \alpha < \omega_1$ and that the statement is valid for all ordinals $\beta<\alpha$. There is nothing to prove
if $\alpha$ is a limit ordinal, so assume that $\alpha=\xi+1$ for some $\xi<\omega_1$.
Fix a sequence $\{\mu_n\}_{n\in \omega}$ in $Seq^\xi(\lin\Delta_K)$ which is $w^*$-convergent to~$\mu$. 
For every $n\in \omega$ we define $\nu_n\in M(K)$ by 
$\nu_n(\Omega):=\int_\Omega \vf\, {\rm d}\mu_n$ for every Borel set $\Omega \sub K$,
so that $\nu_n\in Seq^\xi(\lin\Delta_K)$ by the inductive hypothesis.
Clearly, for every $g\in C(K)$ we have
$$
	\lim_n \int_K g \, {\rm d}\nu_n = \lim_ n\int_K  g\vf \,{\rm d}\mu_n = 
	\int_K  g\vf\, {\rm d}\mu=\int_K g\, {\rm d}\nu,
$$
that is, $\{\nu_n\}_{n\in \omega}$ is $w^*$-convergent to~$\nu$.
Thus $\nu\in Seq^{\xi+1}(\lin\Delta_K)$.
\end{proof}

\begin{lem}\label{seq_abs_cont}
Let $K$ be a compact space and let $\mu\in Seq^{\alpha}({\rm co \Delta_K})$ for some $\alpha<\omega_1$.
If $\nu\in M(K)$ is absolutely continuous with respect to~$\mu$, then 
$\nu\in Seq^{\alpha+1}(\lin\Delta_K)$. If in addition $\nu\in M^+(K)$, then $\nu\in Seq^{\alpha+1}(M^+(K)\cap \lin\Delta_K)$.
\end{lem}
\begin{proof}
Let $\vf:K\to\er$ be the Radon-Nikod\'ym derivative of~$\nu$ with respect to~$\mu$.
Fix a sequence $\{\vf_k\}_{k\in \omega}$ in~$C(K)$
such that $\lim_{k}\int_K |\vf-\vf_k|\,{\rm d}\mu=0$. For every $k\in \omega$
we define $\nu_k\in M(K)$ by 
$\nu_k(\Omega):=\int_\Omega \vf_k\, {\rm d}\mu$ for every Borel set $\Omega \sub K$.
Since each $\nu_k$ belongs to $Seq^\alpha(\lin\Delta_K)$ (by Lemma~\ref{seq_abs_cont2})
and $\{\nu_k\}_{k\in \omega}$ is $w^*$-convergent to~$\nu$ (in fact, it is norm convergent in~$M(K)$),
it follows that $\nu\in Seq^{\alpha+1}(\lin\Delta_K)$. For the last assertion, just
observe that $\varphi$ and the $\varphi_k$'s can be chosen non-negative if $\nu\in M^+(K)$.
\end{proof}

We next fix some terminology and prove further auxiliary results. 
Throughout this section we shall deal with the space $X:=\omega\times 2^\omega$, where $2^\omega=\{0,1\}^\omega$ is the Cantor set. 
For any set $B\sub X$ and $n\in\omega$ we
write $B_{|n}:=\{t\in 2^\omega: (n,t)\in B\}$.
Let $\lambda$ denote the usual product probability 
measure on (the Borel $\sigma$-algebra of) $2^\omega$. 

We will construct an algebra $\alg \sub \mathcal{P}(X)$
such that the Stone space $K=\ult(\alg)$ satisfies the required properties.  
Let $\alg_0$ be the algebra of subsets of~$X$ generated
by the products of the form $A \times C$ where $A \sub \omega$
is either finite or cofinite and $C\in \clop(2^\omega)$. 
Clearly, $\alg_0$ is {\em admissible} in the sense of the following definition:

\begin{definition}\label{ex1:1}
We say that a set $B \sub X$ is {\em admissible} if 
$B_{|n}\in\clop(2^\omega)$ for all $n\in\omega$ and $\lim_n \lambda(B_{|n})$ exists.
In such a case, we write 
$$
	\mu(B):=\lim_n \lambda(B_{|n}).
$$
We say that an algebra $\algb \sub \mathcal{P}(X)$ is {\em admissible} if every $B\in\algb$ is admissible.
\end{definition}

\begin{lemma}\label{ex1:2}
Let $\algb \sub \mathcal{P}(X)$ be a countable admissible algebra and let $\mathcal{D} \sub \algb$. Then
there is a set $A\sub X$ such that: 
\begin{itemize}
\item[(i)] the algebra generated by $\algb\cup\{A\}$ is admissible;
\item[(ii)] for every $D\in \mathcal{D}$ we have
$D_{|n}\sub A_{|n}$ for all but finitely many $n\in \omega$;
\item[(iii)] $\mu(A)\le \sum_{D\in \mathcal{D}} \mu(D)$.
\end{itemize}
\end{lemma}
\begin{proof}
Let $\{B_j:j\in\omega\}$ and $\{D_j:j\in \omega\}$ be enumerations of~$\algb$ and~$\mathcal{D}$,
respectively. For every $k\in \omega$, we denote by~$\algb_k \sub \mathcal{P}(X)$ the finite algebra generated by 
the collection $\{B_j:j\le k\}\cup\{D_j:j\leq k\}$ and we set
$\tilde{D}_k:=D_0\cup\ldots\cup D_k \in \algb_k$.
By the admissibility of~$\algb$ we can define a strictly increasing function $g:\omega\to\omega$ such that
for every $k \in \omega$ and $n \ge g(k)$ we have
\begin{equation}\label{equation:adm}
	\big|\mu(C)-\lambda(C_{|n})\big| \leq \frac{1}{k+1}
	\quad
	\mbox{for all }C\in \algb_k.
\end{equation}
Define a set $A\sub X$ by declaring that 
\[
	A_{|n}:=(\tilde{D}_k)_{|n} \quad\mbox{whenever } g(k)\le n < g(k+1)
\]
and $A_{|n}:=\emptyset$ if $n<g(0)$. Clearly, $A$ satisfies~(ii). 

To prove~(i), notice first that every element~$B$ of the algebra $\algb'$
generated by $\algb\cup\{A\}$ is of the form
$B=(B_j\cap A) \cup (B_i\sm A)$ where $i,j\in \omega$. Since 
$\algb$ is admissible and $A_{|n}\in \clop(2^\omega)$ for every $n\in \omega$, 
we have $B_{|n}\in \clop(2^\omega)$ for every $n\in \omega$. 
To prove the admissibility of~$B$ it suffices to check that
the sequences $\big\{\lambda((B_{j}\cap A)_{|n})\big\}_{n\in \omega}$ and
$\big\{\lambda((B_{i}\cap A)_{|n})\big\}_{n\in \omega}$ are Cauchy, because
$$
	\lambda(B_{|n})=
	\lambda((B_{j}\cap A)_{|n})+
	\lambda((B_{i}\sm A)_{|n})=
	\lambda((B_{j}\cap A)_{|n})+
	\lambda((B_{i})_{|n})-
	\lambda((B_{i}\cap A)_{|n})
$$
for every $n\in \omega$. Fix $\eps>0$. Since $\mu$ is a probability measure on~$\algb$, 
the sequence $\big\{\mu(B_j\cap \tilde{D}_k))\big\}_{k\in \omega}$ is increasing and bounded and 
there is $k_0 \in \omega$ such that
\begin{equation}\label{equation:choicek0}
	\mu(B_j\cap \tilde{D}_k \sm \tilde{D}_{k_0})\leq \epsilon \quad
	\mbox{whenever } k\geq k_0.
\end{equation}
Of course, we can assume further that $k_0 \geq j$ and $\frac{1}{k_0+1} \leq \epsilon$.
Take any $n\ge g(k_0)$. Then $g(k)\le n < g(k+1)$ for some $k\geq k_0$, hence
$(B_{j}\cap A)_{|n}=(B_j \cap \tilde{D}_k)_{|n}$ and so
\begin{multline}\label{multline:1}
	\big|\lambda((B_{j}\cap A)_{|n})-\mu(B_j \cap \tilde{D}_{k_0})\big|=
	\big|\lambda((B_{j}\cap \tilde{D}_k)_{|n})-\mu(B_j \cap \tilde{D}_{k_0})\big| \leq \\
	\leq |\lambda((B_{j}\cap \tilde{D}_k)_{|n})-\mu(B_{j}\cap \tilde{D}_k)\big|+
	\mu(B_j \cap \tilde{D}_{k} \sm \tilde{D}_{k_0})
	\leq \frac{1}{k+1}+\epsilon \leq 2\epsilon,
\end{multline}
by~\eqref{equation:adm} and~\eqref{equation:choicek0}. It follows that
$\big|\lambda((B_{j}\cap A)_{|n})-\lambda((B_{j}\cap A)|_{m})\big|\leq 4\epsilon$
whenever $n,m\geq g(k_0)$. This shows that the sequence $\big\{\lambda((B_{j}\cap A)_{|n})\big\}_{n\in \omega}$ is Cauchy.

Finally, (iii) follows from the argument above by choosing $j\in \omega$ with $B_j=X$. Indeed, by taking limits
in~\eqref{multline:1} when $n\to \infty$ we get $\big|\mu(A)-\mu(\tilde{D}_{k_0})\big| \leq 2\epsilon$
and so 
$$
	\mu(A)\leq 2\epsilon + \mu(\tilde{D}_{k_0}) \leq 2\epsilon + \sum_{i\leq k_0}\mu(D_i)\leq 2\epsilon + \sum_{i\in \omega}\mu(D_i).
$$
As $\epsilon>0$ is arbitrary, we have $\mu(A)\leq \sum_{i\in \omega}\mu(D_i)$ and the proof is over.
\end{proof}

Given an admissible algebra $\algb \sub \mathcal{P}(X)$, we write $\NN(\algb)$ to denote the collection
of all decreasing sequences $\{B_k\}_{k\in \omega}$ in~$\algb$ such
that $\lim_k\mu(B_k)= 0$. 

\begin{lemma}\label{ex1:4}
Let $\algb \sub \mathcal{P}(X)$ be a countable admissible algebra,  
$\mathcal{S} \sub \NN(\algb)$ a countable collection and $\eps>0$.
Then there is a set $A\sub X$ such that:
\begin{itemize}
\item[(i)] the algebra generated by $\algb\cup \{A\}$ is admissible;
\item[(ii)] for every $\{B_k\}_{k\in \omega}\in \mathcal{S}$ there is $k_0\in \omega$ such that
\[ 
	(B_{k_0})_{|n}\sub A_{|n} \quad \mbox{for all but finitely many } n\in \omega;
\]
\item[(iii)] $\mu(A)\leq\eps$.
\end{itemize}
In this case we say that $\mathcal{S}$ is \emph{$\eps$-captured} by~$A$.
\end{lemma}
\begin{proof}
Enumerate $\mathcal{S}=\{\{B^j_k\}_{k\in \omega}: j\in \omega\}$.
For every $j\in \omega$ we can pick $k(j)\in \omega$ such that $\mu(B^j_{k(j)})\leq\eps/2^{j+1}$.
Now it suffices to apply Lemma~\ref{ex1:2} to the collection $\mathcal{D}:=\{B^j_{k(j)}:j\in \omega\}$.
\end{proof}

\begin{lemma}\label{ex1:4a}
Let $\algb \sub \PP(X)$ be an admissible algebra containing~$\alg_0$.
Let $\{\nu_k\}_{k\in \omega}$ be a sequence of probability measures on~$\PP(X)$ such that:
\begin{enumerate}
\item[(i)] each $\nu_k$ is supported by a finite subset of~$X$; 
\item[(ii)] $\lim_k \nu_k(B)=\mu(B)$ for every $B\in\algb$.
\end{enumerate}
Let $B_0\in\algb$ be such that $\mu(B_0)>0$.
Then there is $A\sub B_0$ such that the algebra generated by
$\algb\cup \{A\}$ is admissible and $\{\nu_k(A)\}_{k\in \omega}$ does not converge to $\mu(A)$.
\end{lemma}

\begin{proof} For every $k\in \omega$ we fix a finite set $S_k \sub X$
such that $\nu_k(S_k)=1$. We begin by choosing two strictly increasing sequences in~$\omega$, say
$\{n_j\}_{j\in \omega}$ and $\{k_j\}_{j\in \omega}$, such that
$n_0=k_0=0$ and for every $j\in \omega$ we have:
\begin{enumerate}
\item[(a)] $\nu_{k_{j+1}}(R_j\cap B_0)> \mu(B_0)/2$, where $R_j:=(\omega \sm n_j)\times 2^\omega \in \alg_0$;
\item[(b)] $S_{k_{j+1}}\sub n_{j+1}\times 2^\omega$.
\end{enumerate}
This can be done by induction. Indeed, given $n_j,k_j \in \omega$, the conditions
$$
	\lim_{k}\nu_{k}(R_j)=\mu(R_j)=1 
	\quad
	\mbox{and}
	\quad
	\lim_{k}\nu_k(B_0)= \mu(B_0)>0
$$
ensure the existence of $k_{j+1}>k_j$ for which (a) holds; then we choose $n_{j+1}>n_j$ satisfying~(b)
(bear in mind that $S_{k_{j+1}}$ is finite). 

Fix $n\in \omega$. Take $j\in \omega$ such that $n_j \leq n < n_{j+1}$. 
Since $\lambda$ is atomless and $(B_0 \cap S_{k_{j+1}})_{|n}$ is finite, there is
$C_n\in \clop(2^\omega)$ such that 
\begin{equation}\label{equation:sectionCn}
	(B_0 \cap S_{k_{j+1}})_{|n} \sub C_n \sub (B_0)_{|n} \quad\mbox{and}
\end{equation}
\begin{equation}\label{equation:sectionCn2}
	\lambda(C_n)\leq \frac{1}{j+1}.
\end{equation}

Now, define a set $A \sub B_0$ by declaring that $A_{|n}:=C_n$ for every $n\in \omega$. We claim that $A$ satisfies
the required properties. Note that the algebra $\algb'$
generated by $\algb\cup\{A\}$ is made up of all sets of the form
$(B_1\cap A) \cup (B_2\sm A)$ where $B_1,B_2\in \algb$. 
Thus, since $\algb$ is admissible and $A_{|n}\in \clop(2^\omega)$ for every $n\in \omega$, 
we also have $B_{|n}\in \clop(2^\omega)$ for every $B\in \algb'$ and $n\in \omega$. 
On the other hand, \eqref{equation:sectionCn2}~implies that $\lim_n\lambda(A_{|n})=0$, hence $\mu(A)=0$ and
for any $B_1,B_2\in \algb$ there exists the limit
$$
	\lim_n\lambda\Big(\big((B_1\cap A) \cup (B_2\sm A)\big)_{|n}\Big)=\mu(B_2).
$$
This proves that~$\algb'$ is admissible.

On the other hand, we claim that for every $j\in \omega$ we have 
\begin{equation}\label{equation:inclusionRj}
	R_j \cap B_0 \cap S_{k_{j+1}} \sub A\cap S_{k_{j+1}}.
\end{equation}
Indeed, take $n\in \omega$. If either $n<n_j$ or $n\geq n_{j+1}$, then 
$(R_j \cap B_0 \cap S_{k_{j+1}})_{|n}=\emptyset$ (bear in mind~(b)). If $n_j \leq n < n_{j+1}$, then 
\eqref{equation:sectionCn} implies that $(R_j \cap B_0 \cap S_{k_{j+1}})_{|n} \sub (A\cap S_{k_{j+1}})_{|n}$.
This proves the inclusion~\eqref{equation:inclusionRj}.

It follows that for every $j\in \omega$ we have
\begin{multline*}
	\nu_{k_{j+1}}(A)=
	\nu_{k_{j+1}}(A\cap S_{k_{j+1}})\stackrel{\eqref{equation:inclusionRj}}{\geq}
	\nu_{k_{j+1}}(R_j\cap B_0\cap S_{k_{j+1}})= \\ =
	\nu_{k_{j+1}}(R_j\cap B_0) \stackrel{{\rm (a)}}{>} \frac{\mu(B_0)}{2}>0.
\end{multline*}
Hence the sequence
$\{\nu_k(A)\}_{k\in \omega}$ does not converge to $\mu(A)=0$.
\end{proof}

After those preparations we are ready for our main result.

\begin{theo} \label{ex1:9}
Assuming CH there is a compact space $K$ such that
$$ 
	Seq^1({\rm co}\Delta_K) \neq Seq({\rm co}\Delta_K) = P(K).
$$
\end{theo}
\begin{proof} 
Here we use the terminology and notation introduced in the present section. 

\smallskip
{\sc I. The construction.} 
Let $\{\{\nu^\xi_k\}_{k\in \omega}:\xi<\omega_1\}$
be the collection of all sequences of non-negative measures on~$\PP(X)$ 
which are supported by a finite subset of~$X$.
We shall construct by induction an increasing transfinite collection of countable admissible
algebras $\{\alg_\xi: \xi < \omega_1\}$
of subsets of~$X$. We start from the algebra $\alg_0$ already defined and for any limit ordinal $\xi<\omega_1$ we simply 
set $\alg_\xi:=\bigcup_{\eta<\xi}\alg_\eta$. 

For the successor step of the induction, let $\xi < \omega_1$ and suppose we have already constructed 
the algebras $\{\alg_\eta:\eta \leq \xi\}$. For every $\eta \leq \xi$ we enumerate 
$\NN(\alg_\eta)$ as $\{S(\eta,\alpha):\alpha<\omega_1\}$. 
Lemma~\ref{ex1:4} applied to the countable collection
$$
	\mathcal{S}(\xi):=\{S(\eta,\alpha):\, \eta, \alpha<\xi\} \sub \NN(\alg_\xi)
$$
ensures the existence of a set $A(\xi,2) \sub X$
such that $\mathcal{S}(\xi)$ is $(1/2)$-captured by~$A(\xi,2)$.
Since the algebra generated by $\alg_\xi\cup \{A(\xi,2)\}$ is admissible, we 
can apply again Lemma~\ref{ex1:4} to that algebra to find a set $A(\xi,3) \sub X$
such that $\mathcal{S}(\xi)$ is $(1/3)$-captured by~$A(\xi,3)$
and the algebra generated by $\alg_\xi\cup \{A(\xi,2),A(\xi,3)\}$
is admissible. Continuing in this manner we obtain a sequence 
$\{A(\xi,j):j \geq 2\}$ of subsets of~$X$
such that: 
\begin{enumerate}
\item[(a)] $\mathcal{S}(\xi)$ is $(1/j)$-captured by~$A(\xi,j)$
for all $j\geq 2$;
\item[(b)] the algebra $\bar{\alg}_\xi$ generated by $\alg_\xi$
and the family $\{A(\xi,j):j\ge 2\}$ is admissible. 
\end{enumerate}
We now define a set $A(\xi,1)\sub X$ by distinguishing two cases:
\begin{itemize}
\item If $\lim_k \nu^\xi_k(B)=\mu(B)$ for every $B\in \bar{\alg}_\xi$, 
then we can apply Lemma~\ref{ex1:4a} to find 
a set $D_\xi\subset X\sm A(\xi,2)$ such that the algebra generated by
$\bar{\alg}_\xi\cup \{D_\xi\}$ is admissible and $\{\nu^\xi_k(D_\xi)\}_{k\in \omega}$ does not converge to $\mu(D_\xi)$.
Set $A(\xi,1):=X\sm D_\xi$. 
\item Otherwise, we set $A(\xi,1):=A(\xi,2)$.
\end{itemize}
We now conclude the successor step by letting $\alg_{\xi+1}$ 
be the countable algebra generated by $\alg_\xi$ and the family $\{A(\xi,j):j\ge 1\}$.
Observe that $\alg_{\xi+1}$ is admissible.

Define an admissible algebra $\alg \sub \PP(X)$ by $\alg:=\bigcup_{\xi<\omega_1}\alg_\xi$.
Note that $\alg$ has the following properties:
\begin{itemize}
\item[(i)] for every countable collection $\mathcal{S} \sub \NN(\alg)$ and every $\eps>0$ there is $A\in\alg$
such that $\mathcal{S}$ is $\eps$-captured by~$A$;
\item[(ii)] for every sequence $\{B_k\}_{k\in \omega}\in \NN(\alg)$ there exists $\xi_0<\omega_1$ such that
for every $\xi_0\leq\xi<\omega_1$ and every $j\ge 1$ there is $k\in \omega$ such that 
$(B_k)_{|n}\sub A(\xi,j)_{|n}$ for all but finitely many~$n\in \omega$.
\end{itemize}
Indeed, these facts follow from property~(a) above, bearing in mind that any countable collection~$\mathcal{S}\sub \NN(\alg)$ 
is contained in~$\mathcal{S}(\xi_0)$ for some $\xi_0 < \omega_1$.

\smallskip

{\sc II. Introducing the compact space $K$.}
We now consider the compact space $K=\ult(\alg)$. 
Let $K^* \sub K$ be the set of all ultrafilters that contain no set of the form $\{n\}\times 2^\omega$. 
We claim that every $\FF \in K \sm K^*$ is of the form $\FF_x:=\{A\in \alg:x\in A\}$ for some $x\in X$.
Indeed, if $\FF$ contains $\{n\}\times 2^\omega$ for some $n\in \omega$,
then the collection $\{A_{|n}: A\in \FF\}$ 
is an ultrafilter on~$\clop(2^\omega)$, so the intersection 
$\bigcap \{A_{|n}: A\in \FF\}$ consists of a single point $t\in 2^\omega$, and  
therefore $\FF=\FF_x$ for $x:=(n,t)\in X$. 
Since $K \sm K^*=\bigcup_{n\in \omega} \widehat{\{n\}\times 2^\omega}$, 
the set $K^*$ is closed in~$K$. For any $A,B \in \alg$ we have
\begin{equation}\label{star}
\widehat{B}\cap K^* \sub \widehat{A}\cap K^* \quad \mbox{whenever }
B_{|n} \sub A_{|n} \mbox{ for all but finitely many }n\in \omega.
\end{equation}

Since $\alg$ is admissible, for every $n\in \omega$ we have a probability
measure $\mu_n$ on $\alg$ defined by 
$$
	\mu_n(A):=\lambda(A_{|n})
$$ 
and $\lim_{n} \mu_n(A)= \mu(A)$ for every $A\in \alg$.
Note that $\mu$ (seen as a Radon measure on~$K$) is concentrated on~$K^\ast$, because
$\mu(\{n\}\times 2^\omega)=0$ for every $n\in \omega$. We also have
\begin{equation}\label{equation:atomless}
	\mu(\{\FF\})=0 \quad \mbox{for every }\FF \in K.
\end{equation} 
Indeed, fix $\epsilon>0$ and take any partition $\mathcal{C}$ of~$2^\omega$ into finitely many clopen
sets with $\lambda(C)\leq \epsilon$ for all $C\in \mathcal{C}$. For every $\FF\in K$ there is
some $C\in \mathcal{C}$ such that $\omega \times C \in \FF$  
and so $\mu(\{\FF\})\leq \mu(\omega \times C) = \lambda(C) \leq \epsilon$. 
As $\epsilon>0$ is arbitrary, this proves~\eqref{equation:atomless}.

\smallskip
{\sc III. Claim.} {\em Every closed $\mathcal{G}_\delta$ set $H\sub K^*$ with $\mu(H)=0$ is metrizable.}

Indeed, it is easy to see that we can write 
\begin{equation}\label{equation:H}
	H=\bigcap_{k\in \omega} \widehat{B_k} \cap K^*
\end{equation}
for some $\{B_k\}_{k\in \omega}\in \NN(\alg)$. Now let $\xi_0<\omega_1$ be as in property~I(ii) above.
We shall check that the countable family $\mathcal{A}:=\{\widehat{A} \cap H: A\in\alg_{\xi_0}\}$ is a topological basis of~$H$
(which implies that $H$ is metrizable).

To this end it is sufficient to show that $\widehat{A} \cap H \in \mathcal{A}$ whenever $A \in \alg$. 
We proceed by transfinite induction bearing in mind that $\alg=\bigcup_{\xi <\omega_1}\alg_\xi$. Let $\xi < \omega_1$ and suppose 
that $\widehat{A} \cap H\in \mathcal{A}$ whenever $A \in \bigcup_{\eta < \xi}\alg_\eta$. If either $\xi$ is a limit ordinal
or $\xi \leq \xi_0$ then there is nothing to prove. 
If $\xi$ is of the form $\xi=\eta+1$ for some $\eta \geq \xi_0$, set 
$$
	\algc:=\{A\in \alg_\xi: \, \widehat{A} \cap H\in \mathcal{A}\}.
$$
Observe that $\algc$ is an algebra of subsets of~$X$ containing~$\alg_{\eta}$.
By the choice of~$\xi_0$ (bearing in mind~\eqref{star}), for every $j\ge 1$ there is $k\in \omega$ such that 
$\widehat{B_k} \cap K^* \sub \widehat{A(\eta,j)} \cap K^*$, so 
$H \sub \widehat{A(\eta,j)}$ (by~\eqref{equation:H}), hence $\widehat{A(\eta,j)} \cap H=H\in \mathcal{A}$ and therefore
$A(\eta,j) \in \algc$. It follows that $\alg_\xi=\algc$. 
This proves that $\widehat{A} \cap H \in \mathcal{A}$ whenever $A \in \alg$, as required. 

\smallskip
{\sc IV. Claim.}
{\em For every $j\in \omega$, let $H_j \sub K^*$ be a closed $\mathcal{G}_\delta$
set with $\mu(H_j)=0$. Then $F:=\overline{\bigcup_{j\in \omega} H_j}$ 
is metrizable and $\mu(F)=0$.}

Indeed, as in the previous step, for every $j\in \omega$ 
we can choose $\{B^j_k\}_{k\in \omega}\in \NN(\alg)$
such that
$$
	H_j=\bigcap_{k\in \omega} \widehat{B^j_k} \cap K^*.
$$
Fix $i\in\omega$.
By property I(i), there is $A_i\in \alg$ 
such that the collection $\{\{B^j_k\}_{k\in \omega}:j\in \omega\}$ is $\frac{1}{i+1}$-captured 
by~$A_i$. In view of~\eqref{star}, for every $j\in \omega$ we have 
$$
	H_j=\bigcap_{k\in \omega}\widehat{B^j_k}\cap K^* 
	\sub \widehat{A_i}\cap K^*
$$
and so $F\sub \widehat{A_i}\cap K^*$. Since $\mu(\widehat{A_i})\leq \frac{1}{i+1}$ 
for every $i\in \omega$, we have $\mu(F)=0$. Moreover, 
since $F \sub H:=\bigcap_{i\in \omega}\widehat{A_i}\cap K^*$ and $H$ is metrizable (by Claim~III), it
follows that $F$ is metrizable as well.

\smallskip
{\sc V. Claim.} {\em For every closed separable set $D\sub K^*$ we have $\mu(D)=0$.}

Indeed, let $\{\FF_j:j\in \omega\}$ be a dense sequence in~$D$. For every $j\in \omega$ we have
$\mu(\{\FF_j\})=0$ (by~\eqref{equation:atomless}) and so there is a closed $\mathcal{G}_\delta$
set $H_j \sub K^*$ containing~$\FF_j$ with $\mu(H_j)=0$. Since $D \sub \overline{\bigcup_{j\in \omega} H_j}$, an appeal
to Claim~IV ensures that $\mu(D)=0$.

\smallskip
{\sc VI. Claim.} {\em The measure $\mu$ does not belong to $Seq^1({\rm co}\Delta_K)$.}

Our proof is by contradiction. Suppose there is  
a sequence $\{\theta_k\}_{k\in \omega}$ in~${\rm co}\Delta_K$ which is $w^*$-convergent to~$\mu$. 
For every $k\in \omega$, consider the finite set 
$$
	I_k:=\{\FF\in K^*: \, \theta_k(\{\FF\})>0\}
$$
and let $\theta'_k$ be the Radon measure on~$K$ defined by 
$$
	\theta'_k(\Omega):=\theta_k(\Omega \sm K^*)
	\quad
	\mbox{for every Borel set }\Omega \sub K.
$$

We claim that $\{\theta'_k\}_{k\in \omega}$ is $w^*$-convergent to~$\mu$. Indeed, the set
$D:=\overline{\bigcup_{k\in\omega} I_k}$ satisfies $\mu(D)=0$ (by Claim~V) and so 
we can find $\{B_i\}_{i\in \omega}\in \NN(\alg)$ such that $D \sub \bigcap_{i\in \omega}\widehat{B_i}$.
Now, fix $A\in \alg$. We have
$$
	\big|\theta'_k(\widehat{A})-\theta_k(\widehat{A})\big|=
	\theta_k(\widehat{A} \cap K^*)=
	\theta_k(\widehat{A} \cap D) \leq
	\theta_k(\widehat{B_{i}})
	\quad\mbox{for every }i,k\in \omega.
$$
Bearing in mind that 
$$
	\lim_k \theta_k(\widehat{A})=\mu(\widehat{A}),
	\quad 
	\lim_k \theta_k(\widehat{B_i})=\mu(\widehat{B_i}) 
	\quad\mbox{and}
	\quad
	\lim_i \mu(\widehat{B_i})=0,
$$
we get $\lim_k \theta'_k(\widehat{A})= \mu(\widehat{A})$. As $A\in \alg$ is arbitrary,
$\{\theta'_k\}_{k\in \omega}$ is $w^*$-convergent to~$\mu$.

On the other hand, each $\theta'_k$ is a linear combination (with non-negative coefficients)
of finitely many elements of $\Delta_{K\sm K^*}=\{\delta_{\FF_x}:x\in X\}$, 
where $\FF_x=\{A\in \alg:x\in A\}$ (see~II). Hence $\theta'_k$ comes from a
non-negative finitely supported measure on~$\mathcal{P}(X)$ and so
there is $\xi < \omega_1$ such that $\theta'_k(\widehat{A})=\nu^\xi_k(A)$
for every $A\in \alg$ and $k\in \omega$. By the construction (see~I) there is some $A\in \alg$ such that
$\{\nu^\xi_k(A)\}_{k\in \omega}$ does not converge to~$\mu(A)$, thus contradicting
the fact that $\{\theta'_k\}_{k\in \omega}$ is $w^*$-convergent to~$\mu$. 

\smallskip

{\sc VII. Claim.} {\em The measure $\mu$ belongs to~$Seq^2({\rm co}\Delta_K)$.}

Indeed, the sequence $\{\mu_n\}_{n\in \omega}$ is $w^*$-convergent to~$\mu$ as we pointed out in~II. 
On the other hand, every $\mu_n$ is concentrated on 
the closed metrizable set $\widehat{\{n\}\times 2^\omega}$ 
(it is not difficult to check that it is homeomorphic to~$2^\omega$), hence
$\mu_n$ has a uniformly distributed sequence and so 
$\mu_n\in Seq^1({\rm co}\Delta_K)$.

\smallskip

{\sc VIII. Claim.} {\em The equality $Seq^3({\rm co}\Delta_K) = P(K)$ holds.}

Let $\nu \in P(K)$. We can write
$\nu=\nu_1+\nu_2+\nu_3$ where $\nu_i\in M^+(K)$ satisfy:
\begin{itemize}
\item $\nu_1(\Omega):=\nu(\Omega \sm K^*)$ for every Borel set $\Omega \sub K$;
\item $\nu_2$ is absolutely continuous with respect to $\mu$;
\item $\nu_3$ is concentrated on a Borel set~$B \sub K^*$ with~$\mu(B)=0$.
\end{itemize}

For every $n\in \omega$ we define $\theta_n \in M^+(K)$ by 
$$
	\theta_n(\Omega):=\nu\big(\Omega \cap \widehat{n \times 2^\omega}\big)
		\quad
	\mbox{for every Borel set }\Omega \sub K.
$$
Then $\theta_n\in Seq^1(M^+(K)\cap {\rm span}\Delta_K)$, because
$\theta_n$ is concentrated on the closed metrizable set $\widehat{n \times 2^\omega}=\bigcup_{k<n}\widehat{\{k\}\times 2^\omega}$.
Since the sequence $\{\theta_n\}_{n\in \omega}$ is $w^*$-convergent to~$\nu_1$, we conclude that
$\nu_1 \in Seq^2(M^+(K)\cap {\rm span}\Delta_K)$.

On the other hand, since $\nu_2$ is absolutely continuous with respect to~$\mu\in Seq^2({\rm co}\Delta_K)$ (see 
Claim~VII), we have $\nu_2\in Seq^3(M^+(K)\cap\lin\Delta_K)$ by Lemma~\ref{seq_abs_cont}.

Concerning~$\nu_3$, note that (by the regularity of~$\nu_3$) we can assume that $B$ is of the form $B=\bigcup_{j<\omega} F_j$ for some closed
sets $F_j\sub K^*$. Now, for every $j\in \omega$ we can find (using the regularity of~$\mu$) a closed $\mathcal{G}_\delta$ set 
$H_j \sub K^*$ such that $F_j \sub H_j$ and $\mu(H_j)=0$. 
From Claim IV it follows that $\overline{B}$ is metrizable and so $\nu_3\in Seq^1(M^+(K)\cap\lin\Delta_K)$.

Therefore, $\nu=\nu_1+\nu_2+\nu_3 \in Seq^3(M^+(K)\cap\lin\Delta_K)$. Since
$\nu$ is a probability measure, it is not difficult to prove that
$\nu \in Seq^3({\rm co}\Delta_K)$ as well.
This completes the proof of Theorem~\ref{ex1:9}.
\end{proof}

\section{The cases of $\beta\omega$ and $\beta\omega \sm \omega$}\label{section:betaN}

Let $K$ be a compact space. It is known (cf. \cite[Proposition~3.6]{rod-ver}) that
for every $\Ba(C_p(K))$-measurable $\mu \in P(K)$ there is a 
closed separable set $F\sub K$ such that $\mu(F)=1$. Thus, 
if the equality 
$$
	\Ba(C_p(K))=\Ba(C_w(K))
$$ 
holds true then every element of~$P(K)$
is concentrated on some closed separable subset of~$K$. In this section we make clear
that the converse statement fails in general, since
$\Ba(C_p(\beta\omega))\neq \Ba(C_w(\beta\omega))$ (Theorem~\ref{neq_betaN}).
This will be a consequence of the construction given in Theorem~\ref{measure_on_betaN} below.

Recall that the {\em asymptotic density} of a set $A\sub\omega$ is defined as
$$
	d(A):=\lim_{n}\frac{|A\cap n|}{n}
$$
whenever the limit exists. We shall write $\DD$ for the family of those
$A\sub\omega$ for which $d(A)$ is defined.
The following lemma is well-known.

\begin{lemma}\label{approximation_property}
If $\{A_n\}_{n\in \omega}$ is an increasing sequence in $\DD$, then there is $B\in\DD$ such that
$A_n\sm B$ is finite for every $n\in \omega$ and $d(B)=\lim_n d(A_n)$.
\end{lemma}

\begin{theo}\label{measure_on_betaN} 
There is $\nu \in P(\beta\omega)$ such that: 
\begin{enumerate}
\item[(i)] $\nu$ is of countable type, i.e. $L^1(\nu)$ is separable;
\item[(ii)] $\nu(\beta\omega \sm \omega)=1$;
\item[(iii)] $\nu(F)=0$ for every closed separable set $F\sub \beta\omega\sm \omega$.
\end{enumerate}
\end{theo}
\begin{proof}
Let $\{t_n\}_{n\in \omega}$ be a uniformly distributed sequence for the usual product probability
measure $\lambda$ on~$2^\omega$. For every $E\in \clop(2^\omega)$ we define
$\varphi(E):=\{i\in \omega: t_i\in E\}$, so that 
$$
	\lim_{n}\dfrac{|\varphi(E) \cap n|}{n}=
	\lim_{n}\frac{1}{n}\sum_{i<n}1_{E}(t_i)=
	\lambda(E),
$$
hence $\varphi(E)$ belongs to~$\DD$ and $d(\varphi(E))=\lambda(E)$. It is easy
to check that 
$$
	\alg:=\{\varphi(E):E\in \clop(2^\omega)\} \sub \DD
$$ 
is a (countable) algebra. Let $IS(\alg)$ be the family of all increasing sequences in~$\alg$. 
By Lemma \ref{approximation_property}, for every
$S=\{S_n\}_{n\in \omega}\in IS(\alg)$ we can find $B_S\in \DD$ such that
$S_n\sm B_S$ is finite for every $n\in \omega$ and $d(B_S)=\lim_n d(S_n)$. 

Let $\algb \sub \mathcal{P}(\omega)$ be the algebra generated by~$\alg \cup \{B_S: S\in IS(\alg)\}$.
Fix any free ultrafilter $\mathcal{U}$ on~$\omega$ and
define a probability measure $\mu$ on~$\algb$ by
$$
	\mu(B) := \lim_{n\to\,\mathcal{U}}\frac{|B\cap n|}{n},
$$
so that $\mu(B)=d(B)$ whenever $B\in\algb \cap \DD$. Observe that
the family 
$$
	\algb_0:=\big\{ B\in\algb: \, \inf\{\mu(B\btu A):A\in\alg\}=0 \big\}
$$
is an algebra containing~$\alg$. We claim that 
$B_S\in \algb_0$ for every $S=\{S_n\}_{n\in \omega}\in IS(\alg)$. Indeed, fix $n\in \omega$ and observe that, since $S_n \sm B_S$ is finite
and $S_n\in \DD$, we have $S_n\cap B_S \in \DD$ and $d(S_n\cap B_S)=d(S_n)$.
Now, since $B_S\in \DD$ we also have $B_S \sm S_n\in \DD$ and
$$
	d(B_S \sm S_n)=d(B_S)-d(S_n\cap B_S)=d(B_S)-d(S_n),
$$
hence $\mu(B_S\btu S_n)=
	\mu(S_n \sm B_S)+\mu(B_S \sm S_n)=
	d(B_S)-d(S_n)$.
Bearing in mind that $d(B_S)=\lim_n d(S_n)$, we conclude that
$B_S \in \algb_0$, as required. 

It follows that $\algb_0=\algb$. Using Lemma~\ref{extensions_of_measures}, 
we extend $\mu$ to a probability measure $\nu$ on~$\PP(\omega)$ so that 
$\inf\{\nu(C\btu A):A\in\alg\}=0$ for every $C\sub \omega$. Observe that
$\nu$ (seen as a Radon measure on~$\beta\omega$) has countable type (because $\alg$ is countable).

In order to check that $\nu$ is concentrated on~$\beta\omega \sm \omega$, fix $n\in \omega$ and take any $\epsilon>0$.
Choose a partition $2^\omega=\bigcup_{i=1}^p E_i$ such that each $E_i\in \clop(2^\omega)$ and $\lambda(E_i)\leq \epsilon$.
Then $\omega=\bigcup_{i=1}^p \varphi(E_i)$, each $\varphi(E_i)\in \alg$ and $\nu(\varphi(E_i))=\mu(\varphi(E_i))=d(\varphi(E_i))=\lambda(E_i) \leq \epsilon$.
Since $n\in \varphi(E_i)$ for some~$i$, we have $\nu(\{n\}) \leq \nu (\varphi(E_i)) \leq \epsilon$. As $\epsilon>0$ is arbitrary, 
we get $\nu(\{n\})=0$. It follows that $\nu(\beta\omega\sm \omega)=1$. 

Finally, take any closed separable set $F\sub \beta\omega\sm \omega$ and
let $\{\FF_n\}_{n\in \omega}$ be a dense sequence in~$F$. Fix $\epsilon>0$. 
As in the previous paragraph, for every $k\in \omega$ we can find a partition of~$\omega$ into finitely
many elements of~$\alg$ having asymptotic density 
less than~$\epsilon/2^{k+1}$; one of those elements, say~$T_k$,
belongs to~$\FF_k$. Set 
$$
	S_n:=\bigcup_{k\leq n}T_k
	\quad\mbox{for every }n\in \omega,
$$ 
so that $S=\{S_n\}_{n\in\omega}\in IS(\alg)$. We have $\FF_n\in \widehat{B_S}$ for every $n \in \omega$, because 
$S_n \sm B_S$ is finite and $S_n\in \FF_n \in \beta\omega \sm \omega$. Hence $F \sub \widehat{B_S}$. Since
$$
	d(S_n)\le \sum_{k\leq n} d(T_k) < \sum_{k\leq n} \frac{\epsilon}{2^{k+1}} < \eps
	\quad\mbox{for every }n\in \omega,
$$
it follows that $\nu(F)\leq \nu(\widehat{B_S})=\mu(B_S)=d(B_S)=\lim_{n}d(S_n) \leq \epsilon$. 
As $\epsilon>0$ is arbitrary, we get $\nu(F)=0$. The proof is over.
\end{proof}

Bearing in mind the comments at the beginning of this section,
Theorem~\ref{measure_on_betaN} gives immediately
the following:

\begin{cor}\label{neq_betaNsmN}
$\Ba(C_p(\beta\omega\sm \omega)) \neq \Ba(C_w(\beta\omega\sm \omega))$.
\end{cor}

We arrive at the main result of this section.

\begin{theo}\label{neq_betaN}
$\Ba(C_p(\beta\omega)) \neq \Ba(C_w(\beta\omega))$.
\end{theo}
\begin{proof}
Let $\nu\in P(\beta\omega)$ be the measure of Theorem~\ref{measure_on_betaN}.
We shall prove that $\nu$ is not $\Ba(C_p(\beta\omega))$-measurable by contradiction. Suppose
$\nu$ is $\Ba(C_p(\beta\omega))$-measurable and fix a countable set $I\sub\beta\omega$ such that $\nu$ is measurable
with respect to the $\sigma$-algebra~$\Sigma$ on~$C(K)$ generated by $\{\delta_{\FF}: \FF\in I\}$. 
Set $F:=\overline{I\sm\omega} \sub \beta\omega \sm \omega$, so that
$\nu(F)=0$. Thus, there is $A \sub \omega$ with $\nu(A)>0$ such that 
$\widehat{A}\cap F=\emptyset$.

We can define a measure $m$ on~$\PP(A)$ by 
$m(B):=\nu(B)$ for every $B\sub A$. We claim that $m$ is Borel measurable as a function
on $\PP(A)$ (naturally identified with~$2^A$). Indeed, just observe that the function
$$
	\phi: \PP(A) \to C(K), \quad
	\phi(B):=1_{\widehat{B}},
$$
is Borel-$\Sigma$-measurable, because $\delta_{\FF} \circ \phi=0$
for every $\FF \in I \sm \omega$ (bear in mind that $\widehat{A}\cap (I\sm \omega)=\emptyset$).
Since in addition $m$ vanishes on finite sets and $m(A)>0$, an appeal to \cite[1J]{fre-tal} (cf. \cite[464Q]{freMT-4}) ensures that 
$m$ (seen as a Radon measure on~$\beta A$) has uncountable type, which contradicts
the fact that $\nu$ has countable type.
\end{proof}

\begin{remark}\label{remark:Ppoint}
There is a more direct construction of a Radon probability on $\beta\omega \sm \omega$ 
vanishing on all separable closed subsets which uses a $P$-point.
Recall that a free ultrafilter $\UU$ on~$\omega$ is called a {\em $P$-point} if for every sequence
$\{A_n\}_{n\in \omega}$ in~$\UU$ there is $A\in\UU$ such that $A_n\sm A$ is finite
for every $n\in \omega$. In this case, the measure defined on~$\mathcal{P}(\omega)$ by the formula
$$
	\mu(A) := \lim_{n\to\mathcal{U}}\frac{|A\cap n|}{n}
$$
satisfies the following property (AP) considered in~\cite{Mekler84}: for any increasing sequence $\{A_n\}_{n\in \omega}$ in $\mathcal{P}(\omega)$
there is $A \sub \omega$ such that $A_n\sm A$ is finite for every $n\in \omega$ and 
$\mu(A)=\lim_n\mu(A_n)$; see \cite{BFPR2001} or \cite{Mekler84} for details. 
Therefore, $\mu$ (seen as a Radon measure on $\beta\omega$)
is concentrated on~$\beta\omega \sm \omega$ and vanishes on each separable closed subset of $\beta\omega \sm \omega$
(this can be checked similarly as in the proof of Theorem~\ref{measure_on_betaN}).

However, while $P$-points do exist under Martin's axiom and in
many standard models of ZFC, consistently there are no $P$-points~\cite{Shelah82}. 
Moreover, consistently there are no measures on $\mathcal{P}(\omega)$ extending 
asymptotic density and having property (AP)~\cite{Mekler84}. 
\end{remark}

\section{Some open problems}

In this section we collect several open problems related to the topic of this paper.
Throughout this section $K$ is a compact space.

\begin{problem}
If $\mu\in P(K)$ is $\Ba(C_p(K))$-measurable, does $\mu\in Seq({\rm co}\Delta_K)$?
\end{problem}

\begin{problem}
Is $Seq({\rm co}\Delta_K) = P(K)$ equivalent to $\Ba(C_p(K)) = \Ba(C_w(K))$?
\end{problem}

\begin{problem}
Let $\mu\in Seq^{\alpha}({\rm co \Delta_K})$ for some $\alpha<\omega_1$
and let $\nu\in M(K)$ be absolutely continuous with respect to~$\mu$. Does $\nu\in Seq^{\alpha}(\lin\Delta_K)$?
\end{problem}

\begin{problem}
Does the space $K$ of Theorem~\ref{ex1:9} satisfy $P(K)=Seq^2({\rm co}\Delta_K)$?
\end{problem}

\begin{problem}\label{problem:G}
Do we have $\Ba(C_p(K))\neq \Ba(C_w(K))$ 
whenever $C(K)$ is Grothendieck (and $K$ is infinite)?
\end{problem}

Recall that $C(K)$ is called a {\em Grothendieck} space if every
$w^*$-convergent sequence in~$M(K)$ is necessarily weakly convergent (see e.g. \cite[p.~179]{die-uhl-J}). 
The spaces $C(\beta\omega)$ and $C(\beta\omega \sm \omega)$ are examples of Grothendieck spaces. 
Our motivation for Problem~\ref{problem:G} comes from the results in Section~\ref{section:betaN}
and the following fact:

\begin{prop}\label{Grothendiecknosequential}
If $K$ is infinite and $C(K)$ is a Grothendieck space, then
$$
	Seq({\rm co}\Delta_K) \neq P(K).
$$
\end{prop}

\begin{proof} We first claim that every element of $Seq({\rm co}\Delta_K)$ is concentrated on a countable subset of~$K$. Indeed,
let $\{\mu_n\}_{n\in \omega}$ be any $w^\ast$-convergent sequence in~$P(K)$, where 
each $\mu_n$ is concentrated on a countable set $C_n\sub K$, and 
write $\mu \in P(K)$ to denote its limit.
Since $C(K)$ is Grothendieck, the sequence $\{\mu_n\}_{n\in \omega}$ converges to~$\mu$ weakly in~$M(K)$ and so
$$
	\mu\Big(K\setminus \bigcup_{k\in \omega} C_k\Big) = \lim_n \mu_n\Big(K\setminus \bigcup_{k\in\omega} C_k\Big) = 0,
$$ 
therefore $\mu$ is concentrated on a countable set. This proves the claim.

Since $C(K)$ is Grothendieck, it has no complemented copy of $c_0$ (cf. \cite[p.~179]{die-uhl-J}), 
hence $K$ is not scattered (see e.g. \cite[Theorem~12.30]{fab-alt-J}) and so
there are elements of~$P(K)$ which are not concentrated on a countable subset of~$K$, 
\cite[Theorem 19.7.6]{semadeni}.
It follows that $Seq({\rm co}\Delta_K) \neq P(K)$. 
\end{proof}


\begin{thebibliography}{100}

\bibitem{BFPR2001} A. Blass, R. Frankiewicz, G. Plebanek,  C. Ryll-Nardzewski, 
{\em A note on extensions of asymptotic density}, Proc. Amer. Math. Soc. \textbf{129} (2001), no.~11, 3313--3320.

\bibitem{bor-ple} P.~Borodulin-Nadzieja and G.~Plebanek, 
\emph{On sequential properties of Banach spaces, spaces of measures and densities}, 
Czechoslovak Math. J. \textbf{60(135)} (2010), no.~2, 381--399.

\bibitem{bor-sel} P.~Borodulin-Nadzieja and O.~Selim, 
\emph{Sequential closure in the space of measures}, preprint, arXiv:1203.4270v1.

\bibitem{die-uhl-J} J.~Diestel and J.J. Uhl, Jr., \emph{Vector measures}, 
Mathematical Surveys, No. 15, American Mathematical Society, Providence, R.I., 1977. 

\bibitem{DzamonjaPlebanek08} M. D\v{z}amonja, G. Plebanek, {\em Strictly positive measures on Boolean algebras},  
J. Symbolic Logic {\bf 73} (2008), no.~4, 1416--1432.

\bibitem{edg-J}
G.A. Edgar, \emph{Measurability in a Banach space}, Indiana Univ. Math. J.
  \textbf{26} (1977), no.~4, 663--677. 


\bibitem{fab-alt-J}
M.~Fabian, P.~Habala, P.~H{\'a}jek, V.~Montesinos~Santaluc{\'{\i}}a, J.~Pelant,
  and V.~Zizler, \emph{Functional analysis and infinite-dimensional geometry},
  CMS Books in Mathematics/Ouvrages de Math\'ematiques de la SMC, vol.~8,
  Springer-Verlag, New York, 2001.


\bibitem{freMT-4}
D.H. Fremlin, \emph{Measure theory. {V}ol. 4: Topological measure spaces},
Torres Fremlin, Colchester, 2006.

\bibitem{fre-tal}
D.H. Fremlin and M.~Talagrand, \emph{A decomposition theorem for additive
  set-functions, with applications to {P}ettis integrals and ergodic means},
  Math. Z. \textbf{168} (1979), no.~2, 117--142. 

\bibitem{Ku81} K. Kunen, {\em A compact L-space under CH}, Topology Appl. \textbf{12} (1981), no.~3, 283--287. 

\bibitem{Lipecki74} 
Z. Lipecki, \emph{Extensions of additive set functions with values
in a topological group}, Bull. Acad. Polon. Sci. Ser. Sci. Math. Astron. Phys. 
\textbf{22} (1974), 19--27.




\bibitem{Mekler84} A. Mekler, {\em Finitely additive measures on N and the additive property},  
Proc. Amer. Math. Soc. \textbf{92} (1984), no.~3, 439--444.

\bibitem{mer-J}
S.~Mercourakis, \emph{Some remarks on countably determined measures and uniform
  distribution of sequences}, Monatsh. Math. \textbf{121} (1996), no.~1-2, 79--111. 

\bibitem{nie}
H.~Niederreiter, \emph{On the existence of uniformly distributed sequences in
  compact spaces}, Compositio Math. \textbf{25} (1972), 93--99. 

\bibitem{Plachky76}
D. Plachky, \emph{Extremal and monogenic additive set functions}, Proc. Amer. Math. Soc. \textbf{54} (1976), 193--196.

\bibitem{Pl97} G. Plebanek,
{\em Nonseparable Radon measures and small compact spaces}, 
Fund. Math. \textbf{153} (1997), no.~1, 25--40. 

\bibitem{rod-ver}
J.~Rodr{\'{\i}}guez and G.~Vera, \emph{Uniqueness of measure extensions in
  {B}anach spaces}, Studia Math. \textbf{175} (2006), no.~2, 139--155.

\bibitem{semadeni} 
Z. Semadeni, \emph{Banach spaces of continuous functions}, 
Monografie Matematyczne, Tom 55, PWN - Polish Scientific Publishers, Warsaw, 1971. 

\bibitem{Shelah82} S. Shelah, {\em Proper forcing}, Lecture Notes in Mathematics, 940, Springer-Verlag, Berlin, 1982.


\end{thebibliography}
\end{document}